\documentclass[12pt]{article}
\usepackage[utf8]{inputenc}
\usepackage{url}
\usepackage{hyperref}
\usepackage{amsmath}
\usepackage{graphicx}
\usepackage{float}
\usepackage{longtable}
\usepackage[utf8]{inputenc}
\usepackage{amssymb}
\usepackage{amsfonts}
\usepackage{hyperref}
\usepackage{mathrsfs}
\usepackage{mathpazo}
\usepackage{palatino}
\usepackage[table,xcdraw]{xcolor}
\usepackage{xcolor}
\usepackage{enumerate, textcomp}
\usepackage{cite}
\usepackage{amsthm}
\usepackage{pifont}
\usepackage{hyperref}
\usepackage[square, comma, numbers, sort & compress]{natbib}
\usepackage{interval}
\intervalconfig {soft open fences ,separator symbol =,}

\theoremstyle{plain}
\newtheorem{defn}{Definition}[section]
\newtheorem{pr}{ Proposition}[section]
\newtheorem{lem}{Lemma}[section]
\newtheorem{theorem}{Theorem}[section]
\newtheorem{remark}{Remark}[section]
\newtheorem{cor}{Corollary}[section]
\newtheorem{eg}{Example}[section]

\usepackage{amsmath}
\usepackage{amsfonts}
\usepackage{amssymb}
\usepackage{enumitem}
\date{} 
\usepackage{graphics}
\allowdisplaybreaks

\begin{document}
\pagestyle{myheadings}
	\markboth{{\small\rm \hfill Generalized Proinov- type contractions using simulation functions with applications to fractals
			\hfill}\hspace{-\textwidth}
		\underline{${{}_{}}_{}$\hspace{\textwidth}}}
	{\underline{${{}_{}}_{}$\hspace{\textwidth}}\hspace{-\textwidth}%
		{\small\rm \hfill Generalized Proinov- type contractions using simulation functions with applications to fractals
			\hfill}}
	\thispagestyle{plain}
	
	\begin{center}
		
		{\huge Generalized Proinov- type contractions using simulation functions with applications to fractals 
			\rule{0mm}{6mm}\renewcommand{\thefootnote}{}
			\footnotetext{\scriptsize
				$^*$corresponding author.  e-mail: rameshkumard14@gmail.com\\}
		}
		
		\vspace{1cc}
		{\large\it Athul  Puthusseri $^{a}$, D. Ramesh Kumar$^{b,*}$}
		\vspace{1cc}
		\begin{center}
			$^{a,b}${\small \textit{Department of Mathematics, School of Advanced Sciences, Vellore Institute of Technology,\\
					Vellore-632014, TN, India}}\\
		\end{center}
		\vspace{1cc}
		\begin{abstract}
		The intention of this article is to introduce a generalization of Proinov-type contraction via simulation functions. We name this generalized contraction map as Proinov-type $\mathcal{Z}$-contraction. This article establishes the\textcolor{white}{a}existence and\textcolor{white}{a}uniqueness of\textcolor{white}{a}fixed points for these contraction mappings in quasi-metric space and also, include explanatory examples with graphical interpretation. As an application, we generate a new iterated function system (IFS) consisting of Proinov-type $\mathcal{Z}$-contractions in quasi-metric spaces. At the end of the paper, we prove the existence of a unique attractor for the IFS consisting of Proinov-type $\mathcal{Z}$-contractions.
		\end{abstract}
		
		%

	\end{center}
	\hrule
	
	\vspace{1cc}

\section{Introduction}
The Banach contraction principle is the most famous and widely used fixed point theorem. It was stated and proved by the renowned Polish mathematician Stefan Banach in 1922. Its applications went beyond the boundary of mathematics, to other branches of science, engineering, technology, economics and so on. Many exciting results in fixed point theory came out as extensions of the Banach contraction principle. Recently, in 2020, P. D. Proinov \cite{21} has proved a\textcolor{white}{a}fixed-point result for a\textcolor{white}{a}map $T$ defined on a complete metric space $(X,d)$ to itself, satisfying the\textcolor{white}{a}contraction-type condition. 
\begin{equation}\label{eq:1.1}
    \zeta\left(d\left(Tx, Ty\right)\right)\leq \eta\left(d\left(x, y\right)\right),\ \ \ \text{for\textcolor{white}{a}all } x,y\in X  \text{\textcolor{white}{a}with } d\left(Tx, Ty\right)> 0,
\end{equation}  where $\zeta, \eta :(0,\infty)\rightarrow \mathbb{R}$ are\textcolor{white}{a}two functions which are satisfying the condition $\eta(t)<\zeta(t)$ for $t>0$.\\
The main fixed point result given by P. D. Proinov is:
\begin{theorem}\label{thm:1.1} \cite{21}
Let $(X,d)$ be a\textcolor{white}{a}complete metric space and $T:X\rightarrow X$ be a\textcolor{white}{a}mapping satisfying\textcolor{white}{a}condition (\ref{eq:1.1}), where the functions $\zeta, \eta :(0, \infty)\rightarrow\mathbb{R}$ satisfying the following conditions:
\begin{enumerate}[label=(\roman*)]
\item $\zeta$ is\textcolor{white}{a}nondecreasing;\item $\eta(t)<\zeta(t)$ for any $t>0$;\item  $\limsup\limits_{t\rightarrow \epsilon+}\eta(t)< \zeta(\epsilon+).$
\end{enumerate} Then $T$ has a\textcolor{white}{a}unique fixed point $x^*\in X$ and the\textcolor{white}{t}iterative sequence $\{T^nx\}$ converges to $x^*$ for every $x\in X$.
\end{theorem}
He has shown that this result extends some of the famous fixed point results in the literature, which include Amini- Harandi and Petrusel\cite{1}, Moradi\cite{22}, Geraghty\cite{11}, Jleli and Samet\cite{14}, Wardowski and Van Dung\cite{5}, Secelean\cite{17}, etc. \\In 2015, Khojasteh et al.\cite{8} introduced a new method for the study of fixed points using simulation functions. They have come up with a new kind of contraction map called $\mathcal{Z}$-contractions.
\begin{defn}\label{defn:1.2}\cite{8}
    A\textcolor{white}{a}simulation function is a\textcolor{white}{a}mapping $\xi:[0,\infty)\times[0,\infty)\rightarrow\mathbb{R}$ which satisfies the\textcolor{white}{a}following conditions:
    \begin{enumerate}[start=1,label={(\bfseries $z_\arabic*$)}]
    \item $\xi(0,0)= 0$;
    \item $\xi(s, t)< t-s$\textcolor{white}{a}for\textcolor{white}{a}all $s, t>0$;
    \item for any two sequences $\{s_n\}, \{t_n\}$ in $(0, \infty)$ with the property $\lim\limits_{n\rightarrow\infty}s_n=\lim\limits_{n\rightarrow\infty}t_n>0$, it is true that $\limsup\limits_{n\rightarrow\infty}\xi\left(s_n,t_n\right)<0$.
    \end{enumerate}
\end{defn}
We use the notation $\mathcal{Z}$ to represent the set of all simulation functions. Here are a few illustrations of simulation functions.
\begin{eg}\cite{8}
    Let $\xi_i:[0,\infty)\times[0,\infty)\rightarrow\mathbb{R}$\textcolor{white}{a}for $i=1,2,3$ be\textcolor{white}{a}defined by
    \begin{enumerate}
        \item $\xi_1(s,t)= p(t)-q(s)$ for all $s,t\in [0,\infty)$, where $p,q:[0, \infty)\rightarrow[0,\infty)$\textcolor{white}{a}are continuous functions\textcolor{white}{a}such that $p(t)=q(t)=0$ if and\textcolor{white}{a}only if $t=0$ and $p(t)<t\leq q(t)$ for\textcolor{white}{a}all $t>0$.
        \item $\xi_2(s, t)= t-\frac{f(s,t)}{g(s,t)}s$ for all $s,t\in [0,\infty)$, where $f,g:[0, \infty)\times[0,\infty)\rightarrow[0,\infty)$ are\textcolor{white}{a}continuous\textcolor{white}{a}functions with respect to each\textcolor{white}{a}variable such that $f(s, t)>g(s,t)$ for all $s, t>0$.
        \item $\xi_3(s,t)=t-h(t)-s$ for all $s,t\in[0,\infty)$ where $h:[0,\infty)\rightarrow[0,\infty)$ is a\textcolor{white}{a}continuous function\textcolor{white}{a}satisfying $h(t)=0$ if and only if $t=0$.
    \end{enumerate} Then $\xi_i\in\mathcal{Z}$ for $i=1,2,3$.
\end{eg}
We will define the $\mathcal{Z}$-contraction as follows:
\begin{defn}\label{defn:1.3}\cite{8}
    Let $(X, d)$ be a\textcolor{white}{a}metric space, and $T:X\rightarrow X$. Then $T$ is said to be a $\mathcal{Z}$-contraction with respect to some $\xi\in\mathcal{Z}$ if $\xi\left(d\left(Tx, Ty\right), d(x, y)\right)\geq 0$ for all $x, y\in X$.
\end{defn}
The following Theorem proves that there is a unique fixed point for\textcolor{white}{a}$\mathcal{Z}$-contraction.
\begin{theorem}\cite{8}
    Let $T:X\rightarrow X$ be\textcolor{white}{a}a $\mathcal{Z}$-contraction\textcolor{white}{a}with respect to $\xi\in \mathcal{Z}$, where $(X,d)$ is a complete metric space. Then there exists a\textcolor{white}{a}unique\textcolor{white}{f}fixed point, say $x^*\in X$, of $T$. Furthermore, the iterated sequence $\{T^nx\}$ converges to $x^*$ for every $x\in X$.
\end{theorem}
The quasi-metric is a generalized metric that does not possess the symmetry condition of a metric. This notion was introduced in the literature by W. A. Wilson\cite{23}.
\begin{defn}\label{defn:1.4}\cite{23}
Let $X$ be a nonempty set. Define a function $q:X\times X\rightarrow\mathbb{R}$. Then $q$ is a quasi-metric on $X$ if it satisfies the following conditions:
\begin{enumerate}
    \item $q(x,y)\geq 0$ for every $x,y\in X$.
    \item $q(x, y)= 0$ if and only if $x=y$ for every $x,y\in X$.
    \item $q(x,y)\leq q(x,z)+q(z,y)$ for any $x,y,z\in X$.
\end{enumerate} The set $X$ along with $q$ is called a quasi-metric space and is denoted as $(X,q)$.
\end{defn}
Since there is no symmetry, $q(x,y)$ need not be equal to $q(y,x)$ for any $x,y\in X$. Thus, in quasi-metric spaces, we have two topologies, called forward topology and backward topology. So, concepts such as convergence of sequences, continuity of functions, compactness and completeness got two notions namely forward and backward.

By adding a weaker symmetry condition called $\delta$-symmetry we can get a sub-class of quasi-metric spaces namely, $\delta$-symmetric quasi-metric spaces, which have nicer properties than\textcolor{white}{a}quasi-metric spaces. 
\begin{defn}\label{defn:1.5}
    A\textcolor{white}{a}quasi-metric space $(X,q)$ is said to be a $\delta$ symmetric\textcolor{white}{a}quasi-metric space if there exists\textcolor{white}{a}$\delta>0$ such\textcolor{white}{a}that $q(x,y)\leq \delta q(y,x)$ for all $x,y\in X$.
\end{defn}
In a $\delta$-symmetric quasi-metric space, one can easily observe that forward convergence implies backward convergence and vice versa. 

In this article, we are introducing new types of contraction mappings called $f$-Proinov-type $\mathcal{Z}$-contractions and $b$-Proinov-type $\mathcal{Z}$-contractions in the $\delta$-symmetric quasi-metric space by using simulation functions. We prove the existence\textcolor{white}{a}and uniqueness\textcolor{white}{a}of fixed point for these newly introduced contraction mappings. These fixed point theorems extend to fractal spaces obtained from $\delta$-symmetric quasi-metric space in the last section. We construct an iterated function system consisting of $f$-Proinov-type $\mathcal{Z}$-contractions towards the end of the paper. Further, we prove the existence\textcolor{white}{a}of\textcolor{white}{a}a unique attractor for this iterated function system.

\section{Preliminaries}
This section includes some basic definitions and results in quasi-metric spaces which are required for the further sections of this paper.

Suppose $(X, q)$ is a quasi-metric space. Then it does not need to always be the case that $q(x, y)= q(y, x)$ for $x,y\in X$. So, open balls $B_f(x, r)= \{y\in X: q(x,y)<r\}$ and $B_b(x,r)=\{y\in X: q(y,x)<r\}$, for some $x\in X$ and $r>0$, can be two different sets and are called forward and backward open balls, centered at $x$ with radius $r$, respectively. These two different basic open balls will lead to the following two different topologies in $X$. 
\begin{defn}\label{def:2.1}\cite{23}
    The topology $\tau_f$, whose basis is the collection of all forward open balls $B_f(x, r)= \{y\in X: q(x, y)< r\}$ for $x\in X$ and $r> 0$, on $X$ is called the forward topology.\\ Analogously, the topology $\tau_b$, which has a basis consists of all backward open balls $B_b(x, r)= \{y\in X: q(y, x)< r\}$ for $x\in X$ and $r> 0$, is called the backward topology on $X$.
\end{defn} The following are some examples of\textcolor{white}{a}quasi-metric spaces:
\begin{eg}
  Let $X=\mathbb{R}$ and 
 $q:\mathbb{R}\times\mathbb{R}\rightarrow\mathbb{R}$ be\textcolor{white}{a}defined by 
  \begin{equation*}
      q(\alpha,\beta)=
      \begin{cases}
          \beta-\alpha & \text{if } \beta\geq \alpha\\
          1 & \text{if } \beta<\alpha.
      \end{cases}
  \end{equation*}This $q$ is a quasi-metric on $X$, which is known as\textcolor{white}{a}Sorgenfrey quasi-metric. Here $\tau_f$ is the lower-limit topology and $\tau_b$ is the upper-limit topology on $\mathbb{R}$.
\end{eg}
\begin{eg}
    For any $\lambda>0$, define $q:\mathbb{R}\times\mathbb{R}\rightarrow\mathbb{R}$ by
    \begin{equation*}
        q(\alpha,\beta)=
        \begin{cases}
            \alpha-\beta &\text{if }\alpha\geq \beta\\
            \lambda(\beta-\alpha) &\text{if }\alpha<\beta.
        \end{cases}
    \end{equation*} Here $q$ is a $\lambda$-symmetric quasi-metric space on $\mathbb{R}$. Both the forward and backward topologies here are the usual topology on $\mathbb{R}$.
\end{eg}
These two topologies give rise to two different notions of convergence in the space $X$, namely forward convergence (or $f$-convergence) and backward convergence (or $b$-convergence). Here, $f$-convergence is the convergence in the topology $\tau_f$ and $b$-convergence is the convergence in $\tau_b$. It can be defined in another way as follows: 
\begin{defn}\label{def:2.2}
    Let $\{a_n\}$ be a\textcolor{white}{a}sequence in the quasi-metric\textcolor{white}{a}space $(X,q)$. Then,
    \begin{enumerate}
       \item $\{a_n\}$ is said to be $f$ -converge to $a\in X$ if $q(a, a_n)\rightarrow 0$ as $n\rightarrow\infty$. Then we will write $a_n\xrightarrow[]{f}a$.
        \item $\{a_n\}$ is said to be $b$-converges to $a\in X$ if $q(a_n, a)\rightarrow 0$ as $n\rightarrow\infty$. Then we will write $a_n\xrightarrow[]{b}a$.
    \end{enumerate}
\end{defn}
We have different notions of continuity in quasi-metric spaces since continuity always depends on the underlying topology.
\begin{defn}\label{def;2.3}\cite{16}
    Let $(X,q)$ and $(Y, \rho)$ be two quasi-metric\textcolor{white}{a}spaces. Then a function $g:X\rightarrow Y$ is $ff$-continuous at $x\in X$ if for any\textcolor{white}{a}sequence $x_n\xrightarrow[]{f}x$ in $(X,q)$, one has $g(x_n)\xrightarrow[]{f}g(x)$ in $(Y,\rho)$. Furthermore, $g$ is $ff$-continuous in $X$ if it is $ff$-continuous at each\textcolor{white}{a}point $x\in X$. If $Y=\mathbb{R}$ with the usual topology, then $g$ is said to be $f$-continuous. Analogously, we have other notions of continuities namely, $fb$-continuous, $bf$-continuous, $bb$-continuous and $b$-continuous.
\end{defn}
The next proposition is discussing the continuity of a quasi-metric space.
\begin{pr}\label{prop:2.1}\cite{16}
    If $f$-convergence implies $b$-convergence in a quasi-metric space $(X,q)$, then $q$ is $f$-continuous.
\end{pr}
\begin{remark}\label{rmk:2.1}
Let $\{x_n\}$ is a\textcolor{white}{a}sequence in $(X, q)$, a $\delta$-symmetric\textcolor{white}{a}quasi-metric space. Then $\{x_n\}$ is $f$-convergent if and\textcolor{white}{a}only if it is $b$-convergent in $X$. Therefore, the map $(x,y)\mapsto q(x,y)$ is $f$-continuous.
 \end{remark}
    \begin{proof}
        Suppose that $\{x_n\}$ $f$-converges to $x\in X$. Then we have $\lim\limits_{n\rightarrow\infty}q(x, x_n)= 0$. Since $q$ is $\delta$-symmetric, we have $q(x_n, x)\leq \delta q(x, x_n)$ for all $n\in\mathbb{N}$. Thus, we get $\lim\limits_{n\rightarrow\infty}q(x_n, x)= \delta \lim\limits_{n\rightarrow\infty}q(x, x_n)= 0$, which implies $\{x_n\}$ $b$-converges to $x$. The converse follows in the same way.\\
        The second part follows directly from Proposition\ref{prop:2.1}.
    \end{proof}

Analogous to compactness in metric spaces we have forward and backward compactness in quasi-metric spaces.
\begin{defn}\label{def:2.4}\cite{16}
    A compact subset in the topological space $(X, \tau_f)$ is called a forward compact subset or simply $f$-compact subset of $X$. Similarly, a compact subset in the topological space $(X, \tau_b)$ is called a backward compact or $b$-compact subset of $X$. 
\end{defn}

\section{Main Results}
The results on the existence and uniqueness of fixed points of Proinov-type $\mathcal{Z}$-contractions on quasi-metric spaces are presented in this section. 
\subsection{Auxiliary results}
Here we state some definitions and prove some results that will be used for proving our main theorem.
\begin{defn}\label{defn:3.1}
    Let\textcolor{white}{a}$(X, q)$ be a quasi-metric space. A\textcolor{white}{a}mapping $T: X\rightarrow X$ is said to be forward Proinov-type $\mathcal{Z}$-contraction or $f$-Proinov-type $\mathcal{Z}$-contraction\textcolor{white}{a}with respect to $\xi\in\mathcal{Z}$ if 
    \begin{equation}\label{eq:1}
    \xi\left( \zeta\left(q\left(Tx, Ty\right)\right), \eta\left(q\left(x, y\right)\right)\right)\geq 0
    \end{equation}for\textcolor{white}{a}all $x, y\in X$ where $\zeta, \eta: (0, \infty)\rightarrow \mathbb{R}$ are two control functions with $\eta(t) < \zeta(t)$ for all $t\in Im(q)\setminus\{0\}$.
\end{defn}

\begin{defn}\label{defn:3.2}
    Let $T$ be a\textcolor{white}{a}self-mapping on a quasi-metric space\textcolor{white}{a}$(X, q)$.\textcolor{white}{a}Then $T$ is said to be backward Proinov-type $\mathcal{Z}$-contraction or $b$-Proinov-type $\mathcal{Z}$-contraction\textcolor{white}{a}with\textcolor{white}{a}respect to $\xi\in\mathcal{Z}$ if \begin{equation}\label{eq:2}
    \xi\left( \zeta\left(q\left(Tx, Ty\right)\right), \eta\left(q\left(y, x\right)\right)\right)\geq 0
    \end{equation}for all $x, y\in X$ where $\zeta, \eta: (0, \infty)\rightarrow \mathbb{R}$ are two control functions with $\eta(t) < \zeta(t)$ for all $t\in Im(q)\setminus\{0\}$.
\end{defn}
\begin{pr}\label{prop:3.1}
    An $f$-Proinov-type $\mathcal{Z}$-contraction is both $ff$-continuous and $bb$-continuous if the control function $\zeta$ is nondecreasing.
\end{pr}
\begin{proof}
    Consider a quasi-metric space $(X,q)$ and an $f$-Proinov-type $\mathcal{Z}$-contraction $T:X\rightarrow X$ with respect to the simulation function $\xi$. Let $x\in X$. Consider the sequence $\{x_n\}$ in $X$ which $f$-converges\textcolor{white}{a}to $x$. That is, $q(x,x_n)\rightarrow 0$ as $n\rightarrow\infty$. Then by inequality(\ref{eq:1}) and condition $(z_3)$ in Definition\ref{defn:1.2} we get the following:
    \begin{equation*}
        \begin{split}
            0&\leq\xi\left(\zeta\left(q\left(Tx, Tx_n\right)\right),\eta\left(q\left(x,x_n\right)\right)\right)\\&<\eta\left(q\left(x,x_n\right)\right)-\zeta\left(q\left(Tx, Tx_n\right)\right).
        \end{split}
    \end{equation*}
    This implies $\zeta\left(q\left(Tx, Tx_n\right)\right)<\eta\left(q\left(x,x_n\right)\right)$. Since $\eta(t)<\zeta(t)$ for all $t\in Im(q)\setminus\{0\}$, one can have $\zeta\left(q\left(Tx, Tx_n\right)\right)<\eta\left(q\left(x,x_n\right)\right)<\zeta\left(q\left(x,x_n\right)\right)$. As it is given that $\zeta$ is nondecreasing, we get $q\left(Tx, Tx_n\right)<q\left(x,x_n\right)\rightarrow 0$, which implies $Tx_n\xrightarrow[]{f}Tx$. Hence $T$ is $ff$-continuous.\\Proof of $bb$-continuity follows by a similar argument.
\end{proof}
\begin{pr}\label{prop:3.2}
    A $b$-Proinov-type $\mathcal{Z}$-contraction is both $bf$-continuous and $fb$-continuous if the control function $\zeta$ is nondecreasing.
\end{pr}
\begin{proof}
    The proof is comparable to that of Proposition \ref{prop:3.1}.
\end{proof}
\begin{pr}\label{prop:3.3}
    In a $\delta$-symmetric quasi-metric space $(X,q)$, both $f$-Proinov-type $\mathcal{Z}$-contraction and $b$-Proinov-type $\mathcal{Z}$-contraction satisfy all four types of continuity if the control function $\zeta$ is not decreasing.
\end{pr}
\begin{proof}
    Since $(X,q)$ is $\delta$-symmetric quasi-metric space, we have $f$-convergence implies $b$-convergence and vice versa in $X$. Then the result follows from Propositions \ref{prop:3.1} and \ref{prop:3.2}.
\end{proof}
The notion of\textcolor{white}{a}asymptotic regularity\textcolor{white}{a}was brought into literature by Browder\textcolor{white}{a}and Petryshyn in\cite{7}.
\begin{defn}\label{defn:3.3}\cite{7}
    Let\textcolor{white}{a}$(X, d)$ be a\textcolor{white}{a}metric\textcolor{white}{a}space and $T$ be a self-mapping on $X$. Then $T$\textcolor{white}{a}is said\textcolor{white}{a}to be asymptotically regular at a\textcolor{white}{a}point $x\in X$\textcolor{white}{a}if $\lim\limits_{n\rightarrow\infty}d\left(T^nx, T^{n+1}x\right)= 0$.\\ Furthermore, $T$ is asymptotically regular on $X$ if it is asymptotically regular at each $x\in X$.
\end{defn}
Inspired by this definition, Hamed H. Alsulami et al.\cite{9} introduced the idea of asymptotic regularity in quasi-metric spaces as:
\begin{defn}\label{defn:3.4}\cite{9}
    Let $T$ be a\textcolor{white}{a}self-map on a quasi-metric\textcolor{white}{a}space $(X, q)$. Then $T$ is alleged to be 
    \begin{enumerate}
        \item asymptotically forward regular or asymptotically $f$-regular\textcolor{white}{a}at some point\textcolor{white}{a}$x\in X$ if $\lim\limits_{n\rightarrow\infty}q\left(T^nx, T^{n+1}x\right)= 0$ and\textcolor{white}{a}asymptotically $f$-regular\textcolor{white}{a}on $X$ if it is asymptotically $f$-regular at every point of $X$;
        \item asymptotically backward regular or asymptotically $b$-regular\textcolor{white}{a}at some point\textcolor{white}{a}$x\in X$ if $\lim\limits_{n\rightarrow\infty}q\left(T^{n+1}x, T^{n}x\right)= 0$ and asymptotically\textcolor{white}{a}$b$-regular\textcolor{white}{a}on $X$ if it is\textcolor{white}{a}asymptotically $b$-regular\textcolor{white}{a}at every point of $X$;
        \item asymptotically\textcolor{white}{a}regular if it is both\textcolor{white}{a}asymptotically $f$-regular as well as asymptotically\textcolor{white}{a}$b$-regular.
    \end{enumerate}
\end{defn}
The following lemma provides some conditions for the $f$- Proinov-type $\mathcal{Z}$-contraction to be asymptotically regular.
\begin{lem}\label{lem:3.1}
    Let $T$ be an $f$-Proinov-type $\mathcal{Z}$-contraction with respect to $\xi\in\mathcal{Z}$ on a quasi-metric space $(X, q)$. If the control functions $\zeta \text{ and }\eta$ satisfy\textcolor{white}{a}the following\textcolor{white}{a}conditions:
    \begin{enumerate}[label=(\roman*)]
    \item $\zeta$ is non\textcolor{white}{a}decreasing;
    \item $\eta(t)< \zeta(t)$ for every $t\in Im(q)\setminus\{0\}$;
    \item $\lim\limits_{n\rightarrow\infty}\zeta\left(x_n\right)= \lim\limits_{n\rightarrow\infty}\zeta\left(y_n\right)> 0$ for any two sequences $\{x_n\} \text{ and }\{y_n\}$ in $(0,\infty)$ with $\lim\limits_{n\rightarrow\infty}x_n= \lim\limits_{n\rightarrow\infty}y_n>0$.
    \end{enumerate}
    Then $T$ is asymptotically\textcolor{white}{a}regular in\textcolor{white}{a}$X$.
\end{lem}
\begin{proof}
    Let $x\in X$. Consider the sequence ${T^nx}$. If one can find an $N\in\mathbb{N}$ such that $T^nx= T^{n+1}x$ for every $n\geq N$, then the lemma follows. If not, suppose that $T^nx\neq T^{n+1}x$ for all $n\in\mathbb{N}$. Then, 
    \begin{equation*}
        \begin{split}
            0 &\leq\xi\left(\zeta\left(q\left(T^nx, T^{n+1}x\right)\right), \eta\left(q\left(T^{n-1}x, T^nx\right)\right)\right)\\ &\leq\eta\left(q\left(T^{n-1}x, T^nx\right)\right)- \zeta\left(q\left(T^nx, T^{n+1}x\right)\right).
        \end{split}
    \end{equation*}
    Then by condition $(ii)$ in the hypothesis, we get, $$\zeta\left(q\left(T^nx, T^{n+1}x\right)\right)\leq\eta\left(q\left(T^{n-1}x, T^nx\right)\right)< \zeta\left(q\left(T^{n-1}x, T^nx\right)\right).$$
    From condition $(i)$ in the hypothesis, it follows that $q\left(T^nx, T^{n+1}x\right)\leq q\left(T^{n-1}x, T^nx\right)$. Thus, the sequence $\{q\left(T^nx, T^{n+1}x\right)\}$ is decreasing and bounded below. Hence it converges to a limit, say $r\geq 0$. Let $r>0$. Then we have
    \begin{equation*}
        \begin{split}
            0 &\leq\xi\left(\zeta\left(q\left(T^nx, T^{n+1}x\right)\right), \eta\left(q\left(T^{n-1}x, T^nx\right)\right)\right)\\ &\leq\eta\left(q\left(T^{n-1}x, T^nx\right)\right)- \zeta\left(q\left(T^nx, T^{n+1}x\right)\right)\\ &< \zeta\left(q\left(T^{n-1}x, T^nx\right)\right)- \zeta\left(q\left(T^nx, T^{n+1}x\right)\right).
        \end{split}
    \end{equation*}
    From condition $(iii)$ in the hypothesis, as $n\rightarrow\infty$ we get,  $$\lim\limits_{n\rightarrow\infty}\eta\left(q\left(T^{n-1}x, T^nx\right)\right)=\lim\limits_{n\rightarrow\infty}\zeta\left(q\left(T^nx, T^{n+1}x\right)\right)> 0.$$ Now if we apply condition $(z_3)$ of simulation function, we obtain $$\limsup\limits_{n\rightarrow\infty}\xi\left(\zeta\left(q\left(T^nx, T^{n+1}x\right)\right), \eta\left(q\left(T^{n-1}x, T^nx\right)\right)\right)< 0.$$ This leads to a contradiction. Therefore $r=0$, which proves $T$ is asymptotically $f$-regular. We can demonstrate that $T$ is asymptotically $b$-regular in a similar way. Therefore, it follows that $T$ is asymptotically regular in $X$.
\end{proof}
The next lemma will provide conditions for $b$-Proinov-type $\mathcal{Z}$-contraction to be asymptotically regular. The proof for this lemma differs slightly from the proof for the previous lemma.
\begin{lem}\label{lem:3.2}
    Let\textcolor{white}{a}$T$ be\textcolor{white}{a}a $b$-Proinov-type $\mathcal{Z}$-contraction, on a quasi-metric\textcolor{white}{a}space $(X, q)$,with respect to $\xi\in\mathcal{Z}$. Let the control functions $\zeta, \eta$ follow the conditions:
    \begin{enumerate}[label=(\roman*)]
        \item $\zeta$ is non decreasing;
        \item $\eta(t)< \zeta(t)$ for all $t\in Im(q)\setminus\{0\}$;
        \item if $\{x_n\} \text{ and } \{y_n\}$ are two\textcolor{white}{a}sequences in\textcolor{white}{a}$(0, \infty)$ such\textcolor{white}{a}that $\lim\limits_{n\rightarrow\infty}x_n= \lim\limits_{n\rightarrow\infty}y_n> 0$ then $\lim\limits_{n\rightarrow\infty}\zeta(x_n)= \lim\limits_{n\rightarrow\infty}\zeta(y_n)> 0$.
    \end{enumerate}
    Then $T$\textcolor{white}{a}is asymptotically\textcolor{white}{a}regular in $X$.
\end{lem}
\begin{proof}
    Let $x\in X$. Define $x_n= q(T^nx, T^{n
    +1}x)$. If one can find an $N\in \mathbb{N}$ such that $T^nx= T^{n+1}x$ for all $n\geq N$, then the lemma follows. If not, suppose that $T^nx\neq T^{n+1}x$ for all $n\in\mathbb{N}$. Then,
    \begin{equation*}
        \begin{split}
            0 &\leq \xi\left(\zeta\left(q\left(T^nx, T^{n+1}x\right)\right), \eta\left(q\left(T^nx, T^{n-1}x\right)\right)\right)\\ &\leq\eta\left(q\left(T^nx, T^{n-1}x\right)\right)-\zeta\left(q\left(T^nx, T^{n+1}x\right)\right),
        \end{split}
    \end{equation*}
    which will imply $\zeta\left(q\left(T^nx, T^{n+1}x\right)\right)\leq\eta\left(q\left(T^nx, T^{n-1}x\right)\right)$.
    Then it follows from this and the condition $(ii)$ in the hypothesis that $\zeta\left(q\left(T^nx, T^{n+1}x\right)\right)\leq\eta\left(q\left(T^nx, T^{n-1}x\right)\right)\leq\zeta\left(q\left(T^nx, T^{n-1}x\right)\right)\leq\eta\left(q\left(T^{n-2}x, T^{n-1}x\right)\right)\leq\zeta\left(q\left(T^{n-2}x, T^{n-1}x\right)\right)$. Therefore, from the condition $(i)$ in the hypothesis we get $q\left(T^nx, T^{n+1}x\right)\leq q\left(T^{n-2}x, T^{n-1}x\right)$. i.e., $x_n\leq x_{n-2}$ for all $n\in\mathbb{N}$. This implies that the sequences $\{x_{2n}\} \text{ and }\{x_{2n+1}\}$ are decreasing sequences. We claim that both the sequences $\{x_{2n}\}\text{ and }\{x_{2n+1}\}$ converge to zero. If not, let $x_{2n}\rightarrow r>0$. Then,
    \begin{equation*}
        \begin{split}
            0 &\leq\xi\left(\zeta\left(q\left(T^{2n}x, T^{2n+1}x\right)\right), \eta\left(q\left(T^{2n}x, T^{2n-1}x\right)\right)\right)\\ &\leq \eta\left(q\left(T^{2n}x, T^{2n-1}x\right)\right)- \zeta\left(q\left(T^{2n}x, T^{2n+1}x\right)\right)\\&\leq\zeta\left(q\left(T^{2n}x, T^{2n-1}x\right)\right)- \zeta\left(q\left(T^{2n}x, T^{2n+1}x\right)\right)
        \end{split}
    \end{equation*}
    From condition $(iii)$ in the hypothesis, we get $$\lim\limits_{n\rightarrow\infty}\zeta\left(q\left(T^{2n}x, T^{2n-1}x\right)\right)=\lim\limits_{n\rightarrow\infty} \zeta\left(q\left(T^{2n}x, T^{2n+1}x\right)\right)>0.$$ This implies that $\lim\limits_{n\rightarrow\infty}\eta\left(q\left(T^{2n}x, T^{2n-1}x\right)\right)=\lim\limits_{n\rightarrow\infty}\zeta\left(q\left(T^{2n}x, T^{2n+1}x\right)\right)>0$. Then by condition $(z_3)$ of simulation functions we get, $$\limsup\limits_{n\rightarrow\infty}\xi\left(\zeta\left(q\left(T^{2n}x, T^{2n+1}x\right)\right), \eta\left(q\left(T^{2n}x, T^{2n-1}x\right)\right)\right)< 0,$$ which gives a contradiction. Thus $\{x_{2n}\}$ converges to zero. Similarly, we can prove that $\{x_{2n+1}\}$ also converges to zero. Since both the sequences $\{x_{2n}\} \text{ and }\{x_{2n+1}\}$ decrease and converge to zero, we get $\{x_n\}$ also converges to zero. Hence $T$ is asymptotically $f$-regular. Similarly, we can prove that $T$ is asymptotically $b$-regular and hence it follows that $T$ is asymptotically regular. 
\end{proof}
The following lemmas are crucial for demonstrating our key findings.
\begin{lem}\label{lem:3.3}\cite{2}
    Let $\{x_n\}$ be a sequence such that $\lim\limits_{n\rightarrow\infty}q\left(x_n, x_{n+1}\right)= 0$ in a $\delta$-symmetric quasi-metric space $(X, q)$. If $\{x_n\}$ is not $f$-Cauchy, then one can find an $\epsilon> o$ and two subsequences\textcolor{white}{a}$\{x_{n_k}\} \text{ and\textcolor{white}{a}}\{x_{m_k}\} \text{ of }\{x_n\}$ such\textcolor{white}{a}that $k< m_k< n_k$\textcolor{white}{a}and $\lim\limits_{k\rightarrow\infty}q\left(x_{m_k}, x_{n_k}\right)= \lim\limits_{k\rightarrow\infty}q\left(x_{m_k+1}, x_{n_k}\right)= \lim\limits_{k\rightarrow\infty}q\left(x_{m_k}, x_{n_k+1}\right)= \lim\limits_{k\rightarrow\infty}q\left(x_{m_k+1}, x_{n_k+1}\right)= \epsilon$.
\end{lem}
\begin{lem}\label{lem:3.4}
    Let $\{x_n\}$ be a sequence such that $\lim\limits_{n\rightarrow\infty}q\left(x_n, x_{n+1}\right)= 0$ in a $\delta$-symmetric quasi-metric space $(x, q)$. If the\textcolor{white}{a}sequence $\{x_n\}$ is\textcolor{white}{a}not $f-$Cauchy, then there\textcolor{white}{a}exist $\epsilon> o$ and two\textcolor{white}{a}subsequences $\{x_{n_k}\} \text{ and }\{x_{m_k}\} \text{ of }\{x_n\}$ such\textcolor{white}{a}that $k< m_k< n_k$ and $\lim\limits_{k\rightarrow\infty}q\left(x_{m_k}, x_{n_k}\right)= \lim\limits_{k\rightarrow\infty}q\left(x_{m_k-1}, x_{n_k}\right)= \lim\limits_{k\rightarrow\infty}q\left(x_{m_k-1}, x_{n_k-1}\right)= \epsilon$.
\end{lem}
\begin{proof}
    Since $(X, q)$ is a $\delta$-symmetric quasi-metric space, we can always write $q\left(x_{n+1}, x_n\right)\leq\delta q\left(x_n, x_{n+1}\right)$. Therefore, the\textcolor{white}{a}sequence $\{q\left(x_{n+1}, x_n\right)\}$ will also converge to zero.
    If $\{x_n\}$ is not $f$-Cauchy, then we can find an $\epsilon> 0$ and two\textcolor{white}{a}subsequences $\{x_{m_k}\}\text{ and\textcolor{white}{a}}\{x_{n_k}\}$ of $\{x_n\}$ with $k< m_k< n_k$ such that $q\left(x_{m_k}, x_{n_k}\right)\geq\epsilon$ and $q\left(x_{m_k-1}, x_{n_k}\right)<\epsilon$. Then, $$\epsilon\leq q\left(x_{m_k}, x_{n_k}\right)\leq q\left(x_{m_k}, x_{m_k-1}\right)+q\left(x_{m_k-1}, x_{n_k}\right)< q\left(x_{m_k}, x_{m_k-1}\right)+\epsilon.$$
    Since $\lim\limits_{k\rightarrow\infty}q\left(x_{n+1}, x_n\right)= 0$, we get $$\lim\limits_{k\rightarrow\infty}q\left(x_{m_k}, x_{n_k}\right)= \lim\limits_{k\rightarrow\infty}q\left(x_{m_k-1}, x_{n_k}\right)= \epsilon.$$ Now we have, $$q\left(x_{m_k-1}, x_{n_k}\right)\leq q\left(x_{m_k-1}, x_{n_k-1}\right)+q\left(x_{n_k-1}, x_{n_k}\right)\leq q\left(x_{m_k-1}, x_{n_k}\right)+q\left(x_{n_k}, x_{n_k-1}\right)+q\left(x_{n_k-1}, x_{n_k}\right).$$
    Since $\lim\limits_{k\rightarrow\infty}q\left(x_{m_k-1}, x_{n_k}\right)=\epsilon$ and $\lim\limits_{k\rightarrow\infty}q\left(x_{n_k}, x_{n_k-1}\right)= \lim\limits_{k\rightarrow\infty}q\left(x_{n_k-1}, x_{n_k}\right)= 0$, letting $k\rightarrow\infty$ we get, $\lim\limits_{k\rightarrow\infty}q\left(x_{m_k-1}, x_{n_k-1}\right)=\epsilon$.
\end{proof}
\subsection{Fixed point theorems for forward and backward Proinov-type \texorpdfstring{$\mathcal{Z}$}{Z}-contractions}
We are now ready to demonstrate the\textcolor{white}{a}existence and\textcolor{white}{a}uniqueness of a fixed\textcolor{white}{a}point for $f$-Proinov-type $\mathcal{Z}$-contraction and $b$-Proinov-type $\mathcal{Z}$-contraction.
\begin{theorem}\label{thm:3.1}
    Let\textcolor{white}{a}$(X, q)$ be an $f$-complete\textcolor{white}{a}$\delta$-symmetric quasi-metric\textcolor{white}{a}space. Let\textcolor{white}{a}$T:X\rightarrow X$ be\textcolor{white}{a}a $f$-Proinov-type $\mathcal{Z}$-contraction\textcolor{white}{a}with respect\textcolor{white}{a}to $\xi\in\mathcal{Z}$. If the control functions $\zeta \text{ and }\eta$ follow the below conditions:
    \begin{enumerate}[label=(\roman*)]
    \item $\zeta$ is non decreasing;
    \item $\eta(t)< \zeta(t)$ for\textcolor{white}{a}every $t\in Im(q)\setminus\{0\}$;
    \item if $\{x_n\} \text{ and\textcolor{white}{a}}\{y_n\}$ are\textcolor{white}{a}two sequences\textcolor{white}{a}in $(0,\infty)$ such that $\lim\limits_{n\rightarrow\infty}x_n= \lim\limits_{n\rightarrow\infty}y_n>0$, then $\lim\limits_{n\rightarrow\infty}\zeta\left(x_n\right)= \lim\limits_{n\rightarrow\infty}\zeta\left(y_n\right)> 0$.
    \end{enumerate}
    Then $T$ has\textcolor{white}{a}a unique\textcolor{white}{a}fixed point in $X$. Moreover,\textcolor{white}{a}the iterative\textcolor{white}{a}sequence $\{T^nx\}$ will $f$-converge to the fixed point for any\textcolor{white}{a}$x\in X$.
\end{theorem}
\begin{proof}
    Let $x\in X$. By Lemma \ref{lem:3.1} it is clear that $T$ is asymptotically $f$-regular. Thus, the sequence $\{q\left(T^nx, T^{n+1}x\right)\}$ converges to zero. Define,for each $n\in\mathbb{N}$, $x_n=T^nx$. We claim that $\{x_n\}$ is $f$-Cauchy. If not, then by Lemma \ref{lem:3.3} one can find an $\epsilon> 0$ and subsequences $\{x_{m_k}\}, \{x_{n_k}\}$ of $\{x_n\}$ with $k< m_k< n_k$ such that $\lim\limits_{k\rightarrow\infty}q\left(x_{m_k}, x_{n_k}\right)= \lim\limits_{k\rightarrow\infty}q\left(x_{m_k+1}, x_{n_k+1}\right)= \epsilon> 0$. Then by condition $(iii)$ in the hypothesis we obtain, $$\lim\limits_{k\rightarrow\infty}\zeta\left(q\left(x_{m_k}, x_{n_k}\right)\right)= \lim\limits_{k\rightarrow\infty}\zeta\left(q\left(x_{m_k+1}, x_{n_k+1}\right)\right)> 0.$$ Now, from the contraction condition, we have,
    \begin{equation*}
        \begin{split}
            0 &\leq\xi\left(\zeta\left(q\left(x_{m_k+1}, x_{n_k+1}\right)\right), \eta\left(q\left(x_{m_k}, x_{n_k}\right)\right)\right)\\ &< \eta\left(q\left(x_{m_k}, x_{n_k}\right)\right)- \zeta\left(q\left(x_{m_k+1}, x_{n_k+1}\right)\right)\\ &< \zeta\left(q\left(x_{m_k}, x_{n_k}\right)\right)- \zeta\left(q\left(x_{m_k+1}, x_{n_k+1}\right)\right).
        \end{split}
    \end{equation*} As $k\rightarrow\infty$, by condition $(iii)$ in the hypothesis, we get $$\lim\limits_{k\rightarrow\infty}\zeta\left(q\left(x_{m_k}, x_{n_k}\right)\right)= \lim\limits_{k\rightarrow\infty}\zeta\left(q\left(x_{m_k+1}, x_{n_k+1}\right)\right)> 0.$$ This implies that, $\lim\limits_{k\rightarrow\infty}\eta\left(q\left(x_{m_k}, x_{n_k}\right)\right)= \lim\limits_{k\rightarrow\infty}\zeta\left(q\left(x_{m_k+1}, x_{n_k+1}\right)\right)> 0$.
    Hence from the condition $(z_3)$ of simulation functions we get,$$\limsup\limits_{k\rightarrow\infty}\xi\left(\zeta\left(q\left(x_{m_k+1}, x_{n_k+1}\right)\right), \eta\left(q\left(x_{m_k}, x_{n_k}\right)\right)\right)< 0,$$ which gives a\textcolor{white}{a}contradiction. Thus $\{x_n\}$ is\textcolor{white}{a}$f$-Cauchy. Since $X$\textcolor{white}{a}is $f$-complete\textcolor{white}{a}$\{x_n\}$ will $f$-converge in $X$, say to $w$. Now, we claim that $w$ is a fixed point of $T$. For, we have 
    \begin{equation*}
        \begin{split}
            0 &\leq\xi\left(\zeta\left(q\left(Tw, T^nx\right)\right), \eta\left(q\left(w, T^{n-1}x\right)\right)\right)\\&< \eta\left(q\left(w, T^{n-1}x\right)\right)- \zeta\left(q\left(Tw, T^nx\right)\right),
        \end{split}
    \end{equation*}
    which implies $\zeta\left(q\left(Tw, T^nx\right)\right)< \eta\left(q\left(w, T^{n-1}x\right)\right)$. Now by using condition $(ii)$ followed by $(i)$ from the hypothesis, we get $$\zeta\left(q\left(Tw, T^nx\right)\right)< \eta\left(q\left(w, T^{n-1}x\right)\right)< \zeta\left(q\left(w, T^{n-1}x\right)\right),$$ which implies $q\left(Tw, T^nx\right)< q\left(w, T^{n-1}x\right)$.Then we get, $$0\leq\lim\limits_{n\rightarrow\infty}q\left(Tw, T^nx\right)< \lim\limits_{n\rightarrow\infty}q\left(w, T^{n-1}x\right)= 0,$$ which will imply $Tw= \lim\limits_{n\rightarrow\infty}T^nx= w$. Hence $w$ is a fixed point of $T$.\\ For proving the uniqueness, let $w'\in X$ be another fixed point of $T$. Then,
    \begin{equation*}
        \begin{split}
            0 &\leq\xi\left(\zeta\left(q\left(Tw, Tw'\right)\right),\eta\left(q\left(w, w'\right)\right)\right)\\&< \eta\left(q\left(w, w'\right)\right)- \zeta\left(q\left(Tw, Tw'\right)\right)\\&< \zeta\left(q\left(w, w'\right)\right)- \zeta\left(q\left(Tw, Tw'\right)\right)\\&= \zeta\left(q\left(w, w'\right)\right)- \zeta\left(q\left(w, w'\right)\right)\\&= 0,
        \end{split}
    \end{equation*}
    which gives a contradiction. Thus, the fixed\textcolor{white}{a}point of\textcolor{white}{a}$T$ is unique.
\end{proof}
Next, we\textcolor{white}{a}will give\textcolor{white}{a}an example\textcolor{white}{a}that will illustrate our theorem.
\begin{eg}
    Consider $X=[0, 1]$. Define $q: X\times X\rightarrow \mathbb{R}$ such that:
    \begin{equation*}
    q(x, y)=
        \begin{cases}
            2x & \text{ if } x>y\\
            y & \text{ if } x<y\\
            0 & \text{ if } x=y.
        \end{cases}
    \end{equation*}
    It is easy to see that $q$ is a $2$-symmetric quasi-metric on $X$. Also, $X$ is $f$-complete under the quasi-metric $q$. Define $T: X\rightarrow X$ such that $T(x)= \frac{x^2}{4x^2+3}$. Clearly, $T$ is an increasing map. Also consider the control functions $\zeta, \eta:(0, \infty)\rightarrow \mathbb{R}$ given by $\zeta(t)= t \text{ and } \eta(t)= \frac{t^2}{3}$. Here one can easily verify that both the functions $\zeta\text{ and }\eta$ satisfy the conditions $(i)- (iii)$ in the hypothesis of Theorem \ref{thm:3.1}. Next we define another function $\xi:[0, \infty)\times[0, \infty)\rightarrow\mathbb{R}$ such that $\xi(s, t)= \frac{t}{t+1}- s$. Then $\xi\in \mathcal{Z}$.\\
    \textbf{Case\textcolor{white}{a}1:} If\textcolor{white}{a}$x> y$, then\textcolor{white}{a}$q(x, y)= 2x, T(x)= \frac{x^2}{4x^2+3}\text{ and }T(y)=\frac{y^2}{4y^2+3}$. Since $T$ is increasing, we get $Tx> Ty$. Hence, $q(Tx, Ty)= \frac{2x^2}{4x^2+3}$. Then we have $\zeta\left(q\left(Tx, Ty\right)\right)= q(Tx, Ty)= \frac{2x^2}{4x^2+3} \text{ and }\eta\left(q\left(x, y\right)\right)= \frac{4x^2}{3}$. Therefore, we get the following. 
   \begin{equation*}
     \begin{split}
        \xi\left(\zeta\left(q\left(Tx, Ty\right)\right), \eta\left(q\left(x, y\right)\right)\right)&= \xi\left(\frac{2x^2}{4x^2+3}, \frac{4x^2}{3}\right)\\&=\frac{\frac{4x^2}{3}}{\frac{4x^2}{3}+1}-\frac{2x^2}{4x^2+3}\\&=\frac{4x^2}{4x^2+3}-\frac{2x^2}{4x^2+3}\\&=\frac{2x^2}{4x^2+3}\geq 0
    \end{split}
  \end{equation*}
   \textbf{Case\textcolor{white}{a}2:} If\textcolor{white}{a}$x< y$, then \textcolor{white}{a}$q(x, y)= y$ and $Tx< Ty$. Hence, $q\left(Tx, Ty\right)= Ty= \frac{y^2}{4y^2+3}$. Then $\zeta\left(q\left(Tx, Ty\right)\right)= \frac{y^2}{4y^2+3}\text{ and }\eta\left(q\left(x, y\right)\right)= \frac{y^2}{3}$. Therefore,
   \begin{equation*}
       \begin{split}
           \xi\left(\zeta\left(q\left(Tx, Ty\right)\right), \eta\left(q\left(x, y\right)\right)\right)&= \xi\left(\frac{y^2}{4y^2+3}, \frac{y^2}{3}\right)\\&=\frac{\frac{y^2}{3}}{\frac{y^2}{3}+1}-\frac{y^2}{4y^2+3}\\&=\frac{y^2}{y^2+3}-\frac{y^2}{4y^2+3}\\&=\frac{y^2\left(4y^2+3-\left(y^2+3\right)\right)}{\left(y^2+3\right)\left(4y^2+3\right)}\\&=\frac{3y^2}{\left(y^2+3\right)\left(4y^2+3\right)}\geq 0.
       \end{split}
   \end{equation*}
   \textbf{Case 3:} If $x= y$, then we have $Tx= Ty$ and therefore $q\left(x, y\right)= q\left(Tx, Ty\right)= 0$. Therefore, $$\xi\left(\zeta\left(q\left(Tx, Ty\right)\right), \eta\left(q\left(x, y\right)\right)\right)= \xi(0, 0)= 0.$$
   Hence, in each case, we get $\xi\left(\zeta\left(q\left(Tx, Ty\right)\right), \eta\left(q\left(x, y\right)\right)\right)\geq 0$.
   Thus the map $T$ is an $f$-Proinov- type $\mathcal{Z}$-contraction in $X$. It can be easily observed that $x= 0$ is the unique fixed point of $T$ in $X$.\\ \\ Now we will study the convergence behaviour of the iterated sequence $\{T^n(x_0)\}$ for the map $T$. We will plot the graph of convergence of $\{T^n(x_0)\}$ for different initial ponts $x_0$ in $[0, 1]$. Here we have chosen the points $1,\  0.75,\  0.5 \text{ and } 0.25$ as the initial points. The data used to plot the graph is given in Table 1. Figure \ref{fig1} will display the graph of rate of convergence of $\{T^n(x_0)\}$.
\begin{center}
    \includegraphics[width=1\textwidth]{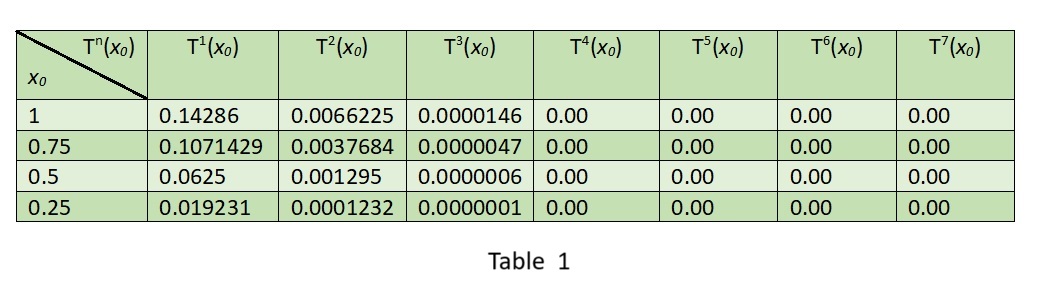}
\end{center}
Here we can observe that, after the third iterate the values of $T^n(x_0)$ is zero or very much close to zero so that we can approximate it to zero. So, as the initial point comes close to zero, the rate of convergence of $\{T^n(x_0)\}$ increases. \\ \\
\begin{figure}[h]
    \centering
    \includegraphics[width=1\textwidth]{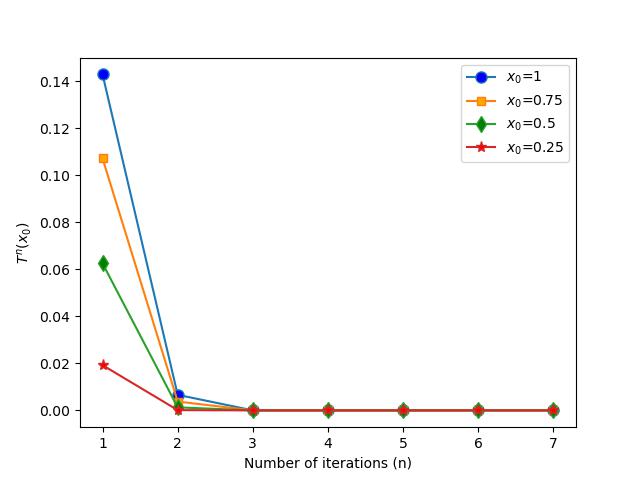}
    \caption{Rate of convergence of the iterated sequence $\{T^n(x)\}$}.
    \label{fig1}
\end{figure}
\end{eg}
Our next theorem will prove the fixed point theorem for $b$-Proinov-type\textcolor{white}{a}$\mathcal{Z}$-contraction.
\begin{theorem}\label{thm:3.2}
    Let\textcolor{white}{a}$(X, q)$ be\textcolor{white}{a}a $\delta$-symmetric quasi-metric\textcolor{white}{a}space and $T$ be a $b$-Proinov-type\textcolor{white}{a}$\mathcal{Z}$-contraction, on $X$, with\textcolor{white}{a}respect to $\xi\in\mathcal{Z}$. Let\textcolor{white}{a}the control functions $\zeta, \eta$ follow the conditions:
    \begin{enumerate}[label=(\roman*)]
        \item $\zeta$ is non decreasing;
        \item $\eta(t)< \zeta(t)$ for\textcolor{white}{a}all $t\in Im(q)\setminus\{0\}$;
        \item if $\{x_n\} \text{ and\textcolor{white}{a}} \{y_n\}$ are\textcolor{white}{a}two sequences\textcolor{white}{a}in $(0, \infty)$ such that $\lim\limits_{n\rightarrow\infty}x_n= \lim\limits_{n\rightarrow\infty}y_n> 0$ then $\lim\limits_{n\rightarrow\infty}\zeta(x_n)= \lim\limits_{n\rightarrow\infty}\zeta(y_n)> 0$.
    \end{enumerate}
    Then $T$\textcolor{white}{a}has a unique\textcolor{white}{a}fixed point\textcolor{white}{a}in $X$,\textcolor{white}{a}provided the space $X$ is $f$-complete. In this case, the sequence $\{T^nx\}$ will $f$-converge to the fixed point for any $x\in X$.
\end{theorem}
\begin{proof}
     Let $x\in X$. By Lemma \ref{lem:3.2}, it is clear that $T$ is asymptotically $f$-regular. Thus, the sequence $\{q\left(T^nx, T^{n+1}x\right)\}$ converges to zero. Let $x_n=T^nx\text{ for each }n\in\mathbb{N}$. We assert that $\{x_n\}$ is $f$-Cauchy. If not, then by Lemma \ref{lem:3.3} and Lemma \ref{lem:3.4} one can find an\textcolor{white}{a}$\epsilon> 0$ and\textcolor{white}{a}subsequences $\{x_{m_k}\}, \{x_{n_k}\}$\textcolor{white}{a}of $\{x_n\}$\textcolor{white}{a}with $k< m_k< n_k$ such\textcolor{white}{a}that $\lim\limits_{k\rightarrow\infty}q\left(x_{m_k}, x_{n_k}\right)= \lim\limits_{k\rightarrow\infty}q\left(x_{m_k+1}, x_{n_k+1}\right)= \lim\limits_{k\rightarrow\infty}q\left(x_{m_k-1}, x_{n_k-1}\right)= \epsilon> 0$. Then by condition $(iii)$ in the hypothesis we obtain, $$\lim\limits_{k\rightarrow\infty}\zeta\left(q\left(x_{m_k-1}, x_{n_k-1}\right)\right)= \lim\limits_{k\rightarrow\infty}\zeta\left(q\left(x_{m_k+1}, x_{n_k+1}\right)\right)> 0.$$ Now, from the contraction condition and condition $(i)$ in the hypothesis we have,
    \begin{equation*}
        \begin{split}
            0 &\leq\xi\left(\zeta\left(q\left(x_{m_k+1}, x_{n_k+1}\right)\right), \eta\left(q\left(x_{n_k}, x_{m_k}\right)\right)\right)\\ &< \eta\left(q\left(x_{n_k}, x_{m_k}\right)\right)- \zeta\left(q\left(x_{m_k+1}, x_{n_k+1}\right)\right)\\ &< \zeta\left(q\left(x_{n_k}, x_{m_k}\right)\right)- \zeta\left(q\left(x_{m_k+1}, x_{n_k+1}\right)\right)\\&\leq\zeta\left(q\left(x_{m_k-1}, x_{n_k-1}\right)\right)- \zeta\left(q\left(x_{m_k+1}, x_{n_k+1}\right)\right).
        \end{split}
    \end{equation*} As $k\rightarrow\infty$, we get,
\begin{equation*}
    \begin{split}
0&\leq\lim\limits_{k\rightarrow\infty}\left(\eta\left(q\left(x_{n_k}, x_{m_k}\right)\right)- \zeta\left(q\left(x_{m_k+1}, x_{n_k+1}\right)\right)\right)\\&< \lim\limits_{k\rightarrow\infty}\left(\zeta\left(q\left(x_{m_k-1}, x_{n_k-1}\right)\right)- \zeta\left(q\left(x_{m_k+1}, x_{n_k+1}\right)\right)\right)= 0.
    \end{split}
\end{equation*} This implies that, $\lim\limits_{k\rightarrow\infty}\eta\left(q\left(x_{n_k}, x_{m_k}\right)\right)= \lim\limits_{k\rightarrow\infty}\zeta\left(q\left(x_{m_k+1}, x_{n_k+1}\right)\right)> 0$.
    Hence from the condition $(z_3)$ of simulation functions, we get,$$\limsup\limits_{k\rightarrow\infty}\xi\left(\zeta\left(q\left(x_{m_k+1}, x_{n_k+1}\right)\right), \eta\left(q\left(x_{n_k}, x_{m_k}\right)\right)\right)< 0,$$ which gives a\textcolor{white}{a}contradiction.\textcolor{white}{a}Thus $\{x_n\}$ is $f$-Cauchy. Since\textcolor{white}{a}$X$ is $f$-complete,\textcolor{white}{a}$\{x_n\}$ will $f$-converge in $X$, say to $w$.\\ The remaining part of the proof mimics the proof of the Theorem \ref{thm:3.1}.
\end{proof}
\section{Application}
\subsection{Fractals Generated by Proinov-type \texorpdfstring{$\mathcal{Z}$}{Z}-contractions}
As an application of our fixed point results, we will extend them to fractal theory.\\
M. F. Barnsley\cite{12,13} mathematically described fractals as fixed points of set-valued maps. The concept of fractals was extended to quasi-metric spaces by Nicolae Adrian Secelean et al.\cite{16}\\ For a quasi-metric space $(X, q)$, we denote by $\mathcal{H}_f(X)$, the collection of\textcolor{white}{a}all nonempty\textcolor{white}{a}$f$-compact\textcolor{white}{a}subsets of\textcolor{white}{a}$X$.\\ For\textcolor{white}{a}two $b$-bounded subsets $A, B$\textcolor{white}{a}of $X$, we define $Q(A, B)= \sup\limits_{x\in A}\inf\limits_{y\in B}q(x, y)$ and $h_q(A,B)= \max\{Q(A, B), Q(B, A)\}$.
\begin{remark}\label{rmk:4.1}\cite{16}
    The condition that $A, B$ to be $b$-bounded is demanded to have $Q(A, B)< \infty$. This inequality may fail if we consider $A, B$ to be $f$-bounded.
\end{remark}
\begin{pr}\label{prop:4.1}\cite{16}
    If $(X, q)$ is a quasi-metric space in which $f$-convergence\textcolor{white}{a}implies\textcolor{white}{a}$b$-convergence, then every\textcolor{white}{a}$f$-compact subset of $X$ is\textcolor{white}{a}$b$-bounded.
\end{pr}

Combining the above fact with Proposition \ref{prop:2.1}, we can have the\textcolor{white}{a}following\textcolor{white}{a}result.
\begin{pr}\label{prop:4.2}
    If $(X, q)$ is a\textcolor{white}{a}$\delta$-symmetric\textcolor{white}{a}quasi-metric\textcolor{white}{a}space, then every $f$-compact subset of $X$ is $b$-bounded.
\end{pr}
\begin{proof}
    By Remark \ref{rmk:2.1}, we have $f$-convergence implies $b$-convergence in $X$. Then the result is immediate from Proposition \ref{prop:4.1}.
\end{proof}
\begin{theorem}\label{thm:4.1}\cite{16}
    If $(X, q)$ is a\textcolor{white}{a}quasi-metric\textcolor{white}{a}space in which a\textcolor{white}{a}sequence is $f$-convergent if and only if it is $b$-convergent, then $(\mathcal{H}_f(X), h_q)$ is a complete metric space.
\end{theorem}
\begin{cor}
    If $(X, q)$ is\textcolor{white}{a}a $\delta$-symmetric\textcolor{white}{a}quasi-metric\textcolor{white}{a}space, then\textcolor{white}{a}$(\mathcal{H}_f(X), h_q)$ is\textcolor{white}{a}a complete metric\textcolor{white}{a}space.
\end{cor}
\begin{proof}
    The proof follows from Remark \ref{rmk:2.1} and Theorem \ref{thm:4.1}.
\end{proof}
\begin{lem}\label{lem:4.1}\cite{16}
    If\textcolor{white}{a}$(X, q)$ is\textcolor{white}{a}a $\delta$-symmetric\textcolor{white}{a}quasi-metric\textcolor{white}{a}space, then\textcolor{white}{a}the metric $h_q$ on $\mathcal{H}_f(X)$ satisfies the following condition:
    $$h_q\left(\bigcup\limits_{i=1}^nA_i, \bigcup\limits_{i=1}^nB_i\right)\leq\max\limits_{1\leq i\leq n}h_q\left(A_i, B_i\right),$$ where $A_i, B_i\in \mathcal{H}_f(X)$ for $i= 1,2, \dots, n$ and $n\in \mathbb{N}.$
\end{lem}
The metric above $h_q$ on $\mathcal{H}_f(X)$ is called the $f$-Hausdorff-Pompeu metric. Here, the complete metric space $(\mathcal{H}_f(X), h_q)$ is called the fractal space. \\Before going to the application, we will prove a fixed\textcolor{white}{a}point theorem\textcolor{white}{a}for Proinov-type $\mathcal{Z}$-contraction\textcolor{white}{a}in complete\textcolor{white}{a}metric space,\textcolor{white}{a}which will be useful further.
First, we will recall a lemma.
\begin{lem}\label{lem:4.2}\cite{21}
    Let $\{x_n\}$ be\textcolor{white}{a}a sequence\textcolor{white}{a}in a\textcolor{white}{a}metric space\textcolor{white}{a}$(X, d)$ such\textcolor{white}{a}that $\lim\limits_{n\rightarrow\infty}d\left(x_n, x_{n+1}\right)= 0$. If $\{x_n\}$ is\textcolor{white}{a}not Cauchy,\textcolor{white}{a}then one can find an $\epsilon> 0$\textcolor{white}{a}and two\textcolor{white}{a}subsequences $\{x_{n_k}\}$ and\textcolor{white}{a}$\{x_{m_k}\}$ such\textcolor{white}{a}that $$\lim\limits_{k\rightarrow\infty}d\left(x_{n_k}, x_{m_k}\right)= \lim\limits_{k\rightarrow\infty}d\left(x_{n_k+1}, x_{m_k+1}\right)= \epsilon.$$
\end{lem}
Now, we can prove the fixed point result for Proinov-type $\mathcal{Z}$-contraction.
\begin{lem}\label{lem:4.3}
    Let\textcolor{white}{a}$T: X\rightarrow X$\textcolor{white}{a}be a\textcolor{white}{a}Proinov-type $\mathcal{Z}$-contraction, on\textcolor{white}{a}a metric space\textcolor{white}{a}$(X, d)$,\textcolor{white}{a}with respect to $\xi\in\mathcal{Z}$. If the control functions $\zeta, \eta$ satisfy\textcolor{white}{a}the following\textcolor{white}{a}conditions:
    \begin{enumerate}[label= (\roman*)]
        \item $\zeta$ is non\textcolor{white}{a}decreasing;
        \item $\eta(t)< \zeta(t)$ for\textcolor{white}{a}all $t\in Im(q)\setminus\{0\}$;
        \item if\textcolor{white}{a}$\{x_n\} \text{ and\textcolor{white}{a}} \{y_n\}$ are two sequences\textcolor{white}{a}in\textcolor{white}{a}$(0, \infty)$ such\textcolor{white}{a}that $\lim\limits_{n\rightarrow\infty}x_n= \lim\limits_{n\rightarrow\infty}y_n> 0$ then $\lim\limits_{n\rightarrow\infty}\zeta(x_n)= \lim\limits_{n\rightarrow\infty}\zeta(y_n)> 0$,
    \end{enumerate} then $T$ is asymptotically regular in $X$.
\end{lem}
\begin{proof}
    The proof mimics the proof of the first part($f$-asymptotic regularity) of Lemma \ref{lem:3.1}
\end{proof}
\begin{theorem}\label{thm:4.2}
     Let\textcolor{white}{a}$T: X\rightarrow X$\textcolor{white}{a}be a\textcolor{white}{a}Proinov-type $\mathcal{Z}$-contraction, on a\textcolor{white}{a}metric space\textcolor{white}{a}$(X, d)$,\textcolor{white}{a}with respect to $\xi\in\mathcal{Z}$.. If the control functions $\zeta, \eta$ follow the conditions:
    \begin{enumerate}[label= (\roman*)]
        \item $\zeta$ is nondecreasing;
        \item $\eta(t)< \zeta(t)$ for\textcolor{white}{a}all $t\in Im(q)\setminus\{0\}$;
        \item if\textcolor{white}{a}$\{x_n\} \text{ and\textcolor{white}{a}} \{y_n\}$ are two sequences\textcolor{white}{a}in $(0, \infty)$ such that $\lim\limits_{n\rightarrow\infty}x_n= \lim\limits_{n\rightarrow\infty}y_n> 0$ then $\lim\limits_{n\rightarrow\infty}\zeta(x_n)= \lim\limits_{n\rightarrow\infty}\zeta(y_n)> 0$,
    \end{enumerate} then $T$\textcolor{white}{a}has a\textcolor{white}{a}unique fixed\textcolor{white}{a}point in\textcolor{white}{a}$X$, say $w$. Moreover, the\textcolor{white}{a}sequence $\{T^nx\}$ converges\textcolor{white}{a}to $w$ for any $x\in X$.
\end{theorem}
\begin{proof}
    Let\textcolor{white}{a}$x\in X$. Then, according to Lemma \ref{lem:4.3}, $T$ is asymptotically regular. Then the sequence $\{d\left(T^nx, T^{n+1}x\right)\}$ converges to zero. Let us denote $x_n= T^nx$ for all $n\in \mathbb{N}$. We\textcolor{white}{a}claim that $\{x_n\}$ is Cauchy. If not, by Lemma \ref{lem:4.2}, there\textcolor{white}{a}exist an\textcolor{white}{a}$\epsilon>0$ and\textcolor{white}{a}two subsequences\textcolor{white}{a}$\{x_{n_k}\}$ and\textcolor{white}{a}$\{x_{m_k}\}$ of\textcolor{white}{a}$\{x_n\}$ such\textcolor{white}{a}that $\lim\limits_{k\rightarrow\infty}d\left(x_{n_k}, x_{m_k}\right)= \lim\limits_{k\rightarrow\infty}d\left(x_{n_k+1}, x_{m_k+1}\right)= \epsilon.$ Then by condition $(iii)$ in the hypothesis, we get $$\lim\limits_{k\rightarrow\infty}\zeta\left(d\left(x_{n_k}, x_{m_k}\right)\right)= \lim\limits_{k\rightarrow\infty}\zeta\left(d\left(x_{n_k+1}, x_{m_k+1}\right)\right)>0.$$
   From the contraction condition of $T$, we get
   \begin{equation*}
       \begin{split}
           0 &\leq \xi\left(\zeta\left(d\left(x_{n_k+1}, x_{m_k+1}\right)\right), \eta\left(d\left(x_{n_k}, x_{m_k}\right)\right)\right)\\ &< \eta\left(d\left(x_{n_k}, x_{m_k}\right)\right)- \zeta\left(d\left(x_{n_k+1}, x_{m_k+1}\right)\right)\\&< \zeta\left(d\left(x_{n_k}, x_{m_k}\right)\right)- \zeta\left(d\left(x_{n_k+1}, x_{m_k+1}\right)\right).
       \end{split}
   \end{equation*}
   Taking the limit $k\rightarrow\infty$ in the above inequality, by condition $(iii)$ in the hypothesis, we get $$\lim\limits_{k\rightarrow\infty}\left(\zeta\left(d\left(x_{n_k}, x_{m_k}\right)\right)- \zeta\left(d\left(x_{n_k+1}, x_{m_k+1}\right)\right)\right)= 0.$$ This implies, $\lim\limits_{k\rightarrow\infty}\eta\left(d\left(x_{n_k}, x_{m_k}\right)\right)= \lim\limits_{k\rightarrow\infty}\zeta\left(d\left(x_{n_k+1}, x_{m_k+1}\right)\right)> 0.$ Hence from the condition $(z_3)$ of\textcolor{white}{a}simulation\textcolor{white}{a}functions,\textcolor{white}{a}we get $$\limsup\limits_{k\rightarrow\infty}\xi\left(\zeta\left(d\left(x_{n_k+1}, x_{m_k+1}\right)\right), \eta\left(d\left(x_{n_k}, x_{m_k}\right)\right)\right)< 0,$$ which contradicts the condition of $\mathcal{Z}$-contraction. Thus $\{x_n\}$\textcolor{white}{a}is a\textcolor{white}{a}Cauchy sequence. Since\textcolor{white}{a}$X$ is\textcolor{white}{a}a complete\textcolor{white}{a}metric space, the\textcolor{white}{a}sequence $\{x_n\}$ will converge in $X$, say to $w\in X$. \\ The\textcolor{white}{a}remaining part\textcolor{white}{a}of the\textcolor{white}{a}proof is\textcolor{white}{a}similar to the proof of\textcolor{white}{a}Theorem \ref{thm:3.1}.
\end{proof}
Let\textcolor{white}{a}$(X, q)$ be\textcolor{white}{a}a $\delta$-symmetric quasi-metric\textcolor{white}{a}space and\textcolor{white}{a}$T: X\rightarrow X$ be\textcolor{white}{a}an $f$-Proinov-type $\mathcal{Z}$-contraction. Define a map $\hat{T}: \mathcal{H}_f(X)\rightarrow \mathcal{P}(X)$\textcolor{white}{a}such that\textcolor{white}{a}$\hat{T}(A)= T(A)= \{T(x): x\in A\}$ for\textcolor{white}{a}$A\in\mathcal{H}_f(X)$. Since $T$ is $ff$-continuous, $T(A)$ will be in $\mathcal{H}_f(X)$. Thus, $\hat{T}$ is a self-mapping of $\mathcal{H}_f(X)$. \\ Next Lemma will prove the map $\hat{T}$ is a Proinov-type $\mathcal{Z}$-contraction on $\mathcal{H}_f(X)$.
\begin{lem}\label{lem:4.4}
    Let\textcolor{white}{a}$(X, q)$ b\textcolor{white}{a} a $\delta$-symmetric quasi-metric\textcolor{white}{a}space and\textcolor{white}{a}$T: X\rightarrow X$ be\textcolor{white}{a}an $f$-Proinov-type $\mathcal{Z}$-contraction\textcolor{white}{a}with respect\textcolor{white}{a}to $\xi\in\mathcal{Z}$, where the simulation function $\xi(s, t)$ decreases on the first variable and increases on the second variable. Suppose that the control functions $\zeta \text{ and } \eta$ are nondecreasing. Then the map $\hat{T}: \mathcal{H}_f(X)\rightarrow\mathcal{H}_f(X)$ defined as $\hat{T}(A)= T(A)$ for $A\in\mathcal{H}_f(X)$ is\textcolor{white}{a}a Proinov-type $\mathcal{Z}$-contraction\textcolor{white}{a}on the\textcolor{white}{a}complete metric\textcolor{white}{a}space\textcolor{white}{a}$(\mathcal{H}_f(X), h_q)$ with\textcolor{white}{a}respect to $\xi\in\mathcal{Z}$.
\end{lem}
\begin{proof}
    Let $A, B\in\mathcal{H}_f(X)$. Then $h_q(A, B)= \max\{Q(A, B), Q(B, A)\}$.\\ Without loss of generality, let $h_q(A, B)= Q(A, B)$. Since $q\text{ and } T$ are continuous and $f$-convergence implies $b$-convergence, there exists $\alpha\in A$ such that 
    \begin{equation*}
        \begin{split}
            \zeta\left(h_q\left(\hat{T}(A), \hat{T}(B)\right)\right)&= \zeta\left(Q\left(\hat{T}(A), \hat{T}(B)\right)\right)\\&= \zeta\left(\inf\limits_{y\in B}q\left(T(\alpha), T(y)\right)\right)\\&\leq \zeta\left(q\left(T(\alpha), T(y)\right)\right),
        \end{split}
    \end{equation*}for any $y\in B$. On the other hand, let\textcolor{white}{a}$\beta\in B$\textcolor{white}{a}be such\textcolor{white}{a}that $q(\alpha, \beta)= \inf\limits_{y\in B}q(\alpha, y)$. Since $\eta$ is increasing, we\textcolor{white}{a}get
    \begin{equation*}
        \begin{split}
            \eta\left(q(\alpha, \beta)\right)&= \eta\left(\inf\limits_{y\in B}q(\alpha, y)\right)\\&\leq \eta\left(\sup\limits_{x\in A}\inf\limits_{y\in B}q(x, y)\right)\\&=\eta\left(Q(A, B)\right)\\&\leq \eta\left(h_q(A, B)\right).
        \end{split}
    \end{equation*}
    That is, we have $\zeta\left(h_q\left(\hat{T}(A), \hat{T}(B)\right)\right)\leq \zeta\left(q\left(T(\alpha), T(\beta)\right)\right) \text{ and }\eta\left(q(\alpha, \beta)\right)\leq \eta\left(h_q\left(A, B\right)\right).$\\ Since the simulation function $\xi$ is decreasing on the first variable and increasing on the second variable, we get
    \begin{equation*}
        \begin{split}
            0&\leq\xi\left(\zeta\left(q\left(T(\alpha), T(\beta)\right)\right), \eta\left(q(\alpha, \beta)\right)\right)\\&\leq \xi\left(\zeta\left(h_q\left(\hat{T}(A), \hat{T}(B)\right)\right), \eta\left(h_q\left(A, B\right)\right)\right).
        \end{split}
    \end{equation*}This implies, $\hat{T}$ is a Proinov-type $\mathcal{Z}$-contraction on $\mathcal{H}_f(X)$.
\end{proof}
\begin{theorem}\label{thm:4.3}
    Let\textcolor{white}{a}$(X, q)$\textcolor{white}{a}be a\textcolor{white}{a}$\delta$-symmetric quasi-metric\textcolor{white}{a}space and\textcolor{white}{a}$T: X\rightarrow X$\textcolor{white}{a}be an\textcolor{white}{a}$f$-Proinov-type $\mathcal{Z}$-contraction\textcolor{white}{a}with respect\textcolor{white}{a}to $\xi\in\mathcal{Z}$. Suppose\textcolor{white}{a}that the\textcolor{white}{a}following conditions\textcolor{white}{a}hold:
    \begin{enumerate}[label=(\roman*)]
        \item $\xi(s, t)$ decreases in\textcolor{white}{a}the first\textcolor{white}{a}variable and increases in\textcolor{white}{a}the second\textcolor{white}{a}variable;
        \item $\zeta, \eta$ are nondecreasing;
        \item $\eta(t)< \zeta(t)$ for\textcolor{white}{a}all $t\in Im(q)\setminus\{0\}$;
        \item if $\{x_n\} \text{ and } \{y_n\}$ are\textcolor{white}{a}two sequences\textcolor{white}{a}in\textcolor{white}{a}$(0, \infty)$ such\textcolor{white}{a}that $\lim\limits_{n\rightarrow\infty}x_n= \lim\limits_{n\rightarrow\infty}y_n> 0$ then $\lim\limits_{n\rightarrow\infty}\zeta(x_n)= \lim\limits_{n\rightarrow\infty}\zeta(y_n)> 0$.
    \end{enumerate}
    Then there exists a unique attractor, say $A^*$ in $\mathcal{H}_f(X)$, for $T$. Moreover the sequence $A_n= T^n(A)$ converges to $A^*$ for any $A\in\mathcal{H}_f(X)$.  
\end{theorem}
\begin{proof}
    By Lemma \ref{lem:4.4}, it\textcolor{white}{a}is clear\textcolor{white}{a}that $\hat{T}$\textcolor{white}{a}is a Proinov-type $\mathcal{Z}$-contraction in\textcolor{white}{a}the complete metric space $\mathcal{H}_f(X)$. Then the result follows from Theorem \ref{thm:4.2}.
\end{proof}
\subsection{Iterated Function System consisting of Proinov-type \texorpdfstring{$\mathcal{Z}$}{Z}-contractions}
Now we will consider an iterated function system (IFS) 
$\{X; w_1, w_2,\dots, w_N\}$ where $N\in\mathbb{N}$ and each $w_i$ is an $f$-Proinov-type $\mathcal{Z}$-contraction. We define a function $W: \mathcal{H}_f(X)\rightarrow\mathcal{H}_f(X)$ by $W(A)= \bigcup\limits_{i=1}^Nw_i(A)$ for any $A\in \mathcal{H}_f(X)$. This map $W$ is called the fractal operator generated by the IFS. A set $A\in\mathcal{H}_f(X)$ that is a fixed point of $W$, that is, $W(A)= \bigcup\limits_{i=1}^Nw_i(A)= A$, is called an attractor of the IFS $\{X; w_1, w_2,\dots, w_N\}$.
The next lemma will show that the fractal operator $W$ defined above is a Proinov-type $\mathcal{Z}$-contraction in $\mathcal{H}_f(X)$. 
\begin{lem}\label{lem:4.5}
    Let\textcolor{white}{a}$(X, q)$ be\textcolor{white}{a}a $\delta$-symmetric\textcolor{white}{a}quasi-metric\textcolor{white}{a}space and\textcolor{white}{a}$w_i: X\rightarrow X$, $i=1, 2,\dots, N$ where $N\in\mathbb{N}$, be $f$-Proinov-type $\mathcal{Z}$-contractions with respect to a simulation function $\xi$ where $\xi(s, t)$ is decreasing on the first variable. If the control functions $\zeta\text{ and } \eta$ are nondecreasing, then the fractal operator $W$, generated by the IFS $\{X; w_1, w_2, \dots, w_N\}$, is a Proinov-type $\mathcal{Z}$-contraction in $\mathcal{H}_f(X)$.  
\end{lem}
\begin{proof}
    Define $W:\mathcal{H}_f(X)\rightarrow \mathcal{H}_f(X)$ by $W(A)= \bigcup\limits_{i=1}^Nw_i(A)$ for any $A\in\mathcal{H}_f(X)$. Since each $w_i$ is an $f$-Proinov-type $\mathcal{Z}$-contraction, by Lemma \ref{lem:4.4} $\hat{w_i}$ is a Proinov-type $\mathcal{Z}$-contraction in $\mathcal{H}_f(X)$. Hence $\xi\left(\zeta\left(h_q\left(\hat{w_i}(A), \hat{w_i}(B)\right)\right), \eta\left(h_q\left(A, B\right)\right)\right)\geq 0$. By Lemma \ref{lem:4.1} we have
    \begin{equation*}
        \begin{split}
            h_q\left(W(A), W(B)\right)&= h_q\left(\bigcup\limits_{i=1}^Nw_i(A), \bigcup\limits_{i=1}^Nw_i(B)\right)\\&\leq\max\limits_{1\leq i\leq N}h_q\left(w_i(A), w_i(B)\right)\\&= h_q\left(w_j(A), w_j(B)\right)\\&= h_q\left(\hat{w_j}(A), \hat{w_j}(B)\right),
        \end{split}
    \end{equation*}for some $j\in \{1, 2, \dots, N\}$. Since $\xi(s, t)$ is decreasing on $s$, we get
    \begin{equation*}
        \begin{split}
            0&\leq \xi\left(\zeta\left(h_q\left(\hat{w_j}(A), \hat{w_j}(B)\right)\right), \eta\left(h_q\left(A, B\right)\right)\right)\\&\leq \xi\left(\zeta\left(h_q\left(W(A), W(B)\right)\right), \eta\left(h_q\left(A, B\right)\right)\right).
        \end{split}
    \end{equation*}Hence the fractal operator $W$ is a Proinov-type $\mathcal{Z}$-contraction.
\end{proof}
The existence and uniqueness of an attractor for an IFS consisting of $f$-Proinov-type $\mathcal{Z}$-contractions are proved in\textcolor{white}{a}the next\textcolor{white}{a}Theorem.
\begin{theorem}\label{thm:4.4}
      Let\textcolor{white}{a}$(X, q)$ be\textcolor{white}{a}a $\delta$-symmetric quasi-metric\textcolor{white}{a}space and $w_i: X\rightarrow X$, $i=1, 2,\dots, N$ where $N\in\mathbb{N}$, be $f$-Proinov-type $\mathcal{Z}$-contractions with respect to a simulation function $\xi$ and  control functions $\zeta\text{ and } \eta$. Suppose\textcolor{white}{a}that the\textcolor{white}{a}following conditions\textcolor{white}{a}hold:
     \begin{enumerate}[label=(\roman*)]
        \item $\xi(s, t)$ is\textcolor{white}{a}decreasing in\textcolor{white}{a}the first\textcolor{white}{a}variable;
        \item $\zeta, \eta$ are nondecreasing;
        \item $\eta(t)< \zeta(t)$ for all $t\in Im(q)\setminus\{0\}$;
        \item if $\{x_n\} \text{ and\textcolor{white}{a}} \{y_n\}$ are\textcolor{white}{a}two sequences\textcolor{white}{a}in $(0, \infty)$ such that $\lim\limits_{n\rightarrow\infty}x_n= \lim\limits_{n\rightarrow\infty}y_n> 0$ then $\lim\limits_{n\rightarrow\infty}\zeta(x_n)= \lim\limits_{n\rightarrow\infty}\zeta(y_n)> 0$.
    \end{enumerate}Then there exists a unique attractor, say $A^*\in\mathcal{H}_f(X)$, for the fractal operator $W$, generated by the IFS $\{X; w_1, w_2, \dots, w_N\}$. Moreover, the iterated sequence $\{W^n(A)\}$ converges to the attractor $A^*$ for any $A\in \mathcal{H}_f(X)$.
\end{theorem}
\begin{proof}
    From Lemma \ref{lem:4.5}, it is clear that the fractal operator $W$ generated by the given IFS is a Proinov-type $\mathcal{Z}$-contraction in the complete metric space $\mathcal{H}_f(X)$. Then the result follows from Theorem \ref{thm:4.2}.
\end{proof}
Next, we will generalize Theorem \ref{thm:4.4}. For that, we will consider an IFS consisting of $f$-Proinov-type $\mathcal{Z}$-contractions each having different simulation functions and control functions. That is, we will take $w_i$ to be $f$-Proinov-type $\mathcal{Z}$-contraction with respect to $\xi_i\in\mathcal{Z}$ and control functions $\zeta\text{ and }\eta_i$ for each $i= 1, 2,\dots,N$.\\ Before moving to the main results, we will prove the following Lemma about simulation functions.
\begin{lem}\label{4.6}
    Let $\xi_i$, for\textcolor{white}{a}$i=1, 2, \dots, N$\textcolor{white}{a}where\textcolor{white}{a}$N\in\mathbb{N}$, be a\textcolor{white}{a}finite collection of simulation\textcolor{white}{a}functions. Define $\xi(s, t)= \max\limits_{1\leq i\leq N}\xi_i(s, t)$. Then the function $\xi$ is also a simulation function.
\end{lem}
\begin{proof}
    From the definition of $\xi(s, t)$, it is clear that $\xi$ is a map from $[0, \infty)\times [0, \infty)$ to $\mathbb{R}$. Since $\xi_i(0, 0)= 0$ for all $i$, we get $\xi(0, 0)= 0$. We have $\xi_i(s, t)< t-s$ for all $s,t>0$ and $i= 1, 2, \dots,N$. Thus, it is clear that $\xi(s, t)= \max\limits_{1\leq i\leq N}\xi_i(s, t)< t-s$ for all $s,t> 0$. So, $\xi$ satisfies the properties $(z_1) \text{ and }(z_2)$ of the simulation function. Now we have to prove the property $(z_3)$. Let $\{s_n\}, \{t_n\}$ be two sequences in $(0, \infty)$ such that $\lim\limits_{n\rightarrow\infty}s_n= \lim\limits_{n\rightarrow\infty}t_n> 0$. Then we have $\limsup\limits_{n\rightarrow\infty}\xi_i(s_n, t_n)< 0$. We claim $\limsup\limits_{n\rightarrow\infty}\xi(s_n, t_n)< 0$. We will prove this by mathematical induction\textcolor{white}{a}on\textcolor{white}{a}$N$. The\textcolor{white}{a}case $N= 1$ is\textcolor{white}{a}trivial.
\textcolor{white}{a}For $N= 2$, let $\xi(s, t)= \max\{\xi_1(s, t), \xi_2(s, t)\}$. Let $a_n= \xi_1(s_n, t_n)$,\   $b_n= \xi_2(s_n, t_n)$ and $c_n= \xi(s_n, t_n)$. Then we have three real sequences $\{a_n\}, \{b_n\}$ and $\{c_n\}$ such that $c_n= \max\limits_{n\in\mathbb{N}}\{a_n, b_n\}$. Let $c= \limsup\limits_{n\rightarrow\infty}c_n$. Then there exists a subsequence $\{c_{n_k}\}$ of $\{c_n\}$ such that $\lim\limits_{k\rightarrow\infty}c_{n_k}= c$. We have three possibilities for $c_{n_k}$:\\
    Case 1: There\textcolor{white}{a}exists $K\in\mathbb{N}$\textcolor{white}{a}such that\textcolor{white}{a}$c_{n_k}= a_{n_k}$ for\textcolor{white}{a}each\textcolor{white}{a}$k\geq K$. Then we get $\lim\limits_{k\rightarrow\infty}a_{n_k}= c.$ This implies $c\leq \limsup\limits_{n\rightarrow\infty}a_n$.\\ Case 2: There exist $K\in\mathbb{N}$ such that $c_{n_k}= b_{n_k}$ for each $k\geq K$. Then, by an argument similar to that in Case 1, we get $c\leq \limsup\limits_{n\rightarrow\infty}b_n$.\\
    Case 3: For each $i\in\mathbb{N}$ there exist $n_{k_i}, n_{l_i}> i$ such that $c_{n_{k_i}}= a_{n_{k_i}}$ and $c_{n_{l_i}}= b_{n_{l_i}}$. That is, there exist two subsequences $\{c_{n_{k_i}}\}$ and $\{c_{n_{l_i}}\}$ of $\{c_{n_k}\}$ such that $c_{n_{k_i}}= a_{n_{k_i}}$ and $c_{n_{l_i}}= b_{n_{l_i}}$. Thus, $\lim\limits_{i\rightarrow\infty}a_{n_{k_i}}= \lim\limits_{i\rightarrow\infty}b_{n_{k
    l_i}}= c$. Hence $c\leq \limsup\limits_{n\rightarrow\infty}a_n$ and $c\leq \limsup\limits_{n\rightarrow\infty}b_n$. \\ In each case we get $c\leq\max\{\limsup\limits_{n\rightarrow\infty}a_n, \limsup\limits_{n\rightarrow\infty}b_n\}$. \\Now, suppose the result is true for $N$. Suppose that $\xi(s, t)= \max\limits_{1\leq i\leq N+1}\xi_i(s, t)$. Then,
    \begin{align*}
    \limsup\limits_{n\rightarrow\infty}\xi(s_n, t_n)&=\limsup\limits_{n\rightarrow\infty}\max\limits_{1\leq i\leq N+1}\xi_i(s_n, t_n)\\&= \limsup\limits_{n\rightarrow\infty}\left(\max\left\{\max\limits_{1\leq i\leq N}\xi_i(s_n, t_n), \xi_{N+1}(s_n, t_n)\right\}\right)\\&\leq \max\left\{\limsup\limits_{n\rightarrow\infty}\max\limits_{1\leq i\leq N}\xi_i(s_n, t_n), \limsup\limits_{n\rightarrow\infty}\xi_{N+1}(s_n, t_n)\right\}\\&\leq \max\left\{\max\limits_{1\leq i\leq N}\limsup\limits_{n\rightarrow\infty}\xi_i(s_n, t_n), \limsup\limits_{n\rightarrow\infty}\xi_{N+1}(s_n, t_n) \right\}\\&= \max\limits_{1\leq i \leq N+1}\limsup\limits_{n\rightarrow\infty}\xi_i(s_n, t_n)\\&<0.
    \end{align*}
    Hence the result is true for any $N\in\mathbb{N}$. Thus, $\xi$ satisfies property $(z_3)$. Therefore, it is a simulation function.
\end{proof}
Next, we will prove a lemma that will generalize the fractal operator given in Lemma \ref{lem:4.4}. 
\begin{lem}\label{lem:4.7}
    Let $\{X; w_1, w_2, \dots,w_N\}$ be an IFS where each $w_i$ is a $f$-Proinov-type $\mathcal{Z}$-contractions with respect to the simulation function $\xi_i$ and control functions $\zeta$ and $\eta_i$. That is, each $w_i$ satisfies\textcolor{white}{a}the contraction\textcolor{white}{a}condition:$$0\leq\xi_i\left(\zeta\left(q\left(w_i(x), w_i(y)\right)\right), \eta_i\left(q\left(x, y\right)\right)\right).$$ Suppose that each simulation function $\xi_i$ decreases on the first variable and increases on the second variable. Also, let the control functions $\zeta\text{ and }\eta_i$ not decrease for $i= 1, 2, \dots, N$. Then the fractal operator $W$ generated by the IFS $\{X; w_1, w_2,\dots,w_N\}$ is a Proinov-type $\mathcal{Z}$-contraction in $\mathcal{H}_f(X)$ with respect to the simulation function $\xi(s, t)=\max\limits_{1\leq i\leq N}\xi_i(s, t)$ and control functions $\zeta, \eta$ where $\eta(t)= \max\limits_{1\leq i\leq N}\eta_i(t)$.
\end{lem}
\begin{proof}
    Define $W:\mathcal{H}_f(X)\rightarrow\mathcal{H}_f(X)$ by $W(A)= \bigcup\limits_{i=1}^Nw_i(A)$ for $A\in \mathcal{H}_f(X)$. Let $A, B \in \mathcal{H}_f(X)$. For each $i=1, 2,\dots, N$, we have $0\leq \xi_i\left(\zeta\left(q\left(w_i(x), w_i(y)\right)\right), \eta_i\left(q\left(x, y\right)\right)\right)$ for any $x, y\in X$. By Lemma \ref{lem:4.4} we get $0\leq \xi_i\left(\zeta\left(h_q\left(\hat{w_i}(A), \hat{w_i(B)}\right)\right), \eta_i\left(h_q\left(A, B\right)\right)\right).$ From the proof of Lemma \ref{lem:4.5}, we have $h_q\left(W(A), W(B)\right)\leq h_q\left(\hat{w_j}(A), \hat{w_j}(B)\right)$ for some $j\in\{1,2,\dots, N\}$. Since each $\xi_i$ decreases in the first variable and increases in the second variable, the function $\xi(s, t)= \max\limits_{1\leq i\leq N}\xi_i(s, t)$ also decreases in\textcolor{white}{a}the first\textcolor{white}{a}variable and\textcolor{white}{a}increases in\textcolor{white}{a}the second variable. Then,
    \begin{align*}
        0&\leq\xi_j\left(\zeta\left(h_q\left(\hat{w_j}(A), \hat{w_j}(B)\right)\right), \eta_j\left(h_q\left(A, B\right)\right)\right)\\&\leq\xi_j\left(\zeta\left(h_q\left(W(A), W(B)\right)\right), \eta_j\left(h_q\left(A, B\right)\right)\right)\\&\leq\xi_j\left(\zeta\left(h_q\left(W(A), W(B)\right)\right), \eta\left(h_q\left(A, B\right)\right)\right)\\&\leq\xi\left(\zeta\left(h_q\left(W(A), W(B)\right)\right), \eta\left(h_q\left(A, B\right)\right)\right).
    \end{align*}Thus the fractal operator $W$ generated by the IFS is a Proinov-type $\mathcal{Z}$-contraction in $\mathcal{H}_f(X)$.
\end{proof}
The next Theorem will prove the existence and uniqueness of attractor for this generalized IFS of $f$-Proinov-type $\mathcal{Z}$-contractions. 
\begin{theorem}\label{thm:4.5}
    Let $\{X; w_1, w_2, \dots,w_N\}$ be an IFS where each $w_i$ is a $f$-Proinov-type $\mathcal{Z}$-contractions with respect to the simulation function $\xi_i$ and control functions $\zeta$ and $\eta_i$. That is, each $w_i$ satisfies\textcolor{white}{a}the contraction\textcolor{white}{a}condition:$$0\leq\xi_i\left(\zeta\left(q\left(w_i(x), w_i(y)\right)\right), \eta_i\left(q\left(x, y\right)\right)\right).$$ Suppose that the following conditions hold:
    \begin{enumerate}[label=(\roman*)]
        \item Each $\xi_i(s, t)$ decreases in\textcolor{white}{a}the first\textcolor{white}{a}variable and\textcolor{white}{a}increases in\textcolor{white}{a}the second variable\textcolor{white}{a}for $i=1, 2,\dots, N$;
        \item $\zeta, \eta_i$ are nondecreasing for each $i=1,2,\dots,N$;
        \item $\eta_i(t)< \zeta(t)$ for all $t\in Im(q)\setminus\{0\}$ and $i=1, 2,\dots, N$;
        \item if $\{x_n\} \text{ and } \{y_n\}$ are two sequences in $(0, \infty)$ such that $\lim\limits_{n\rightarrow\infty}x_n= \lim\limits_{n\rightarrow\infty}y_n> 0$ then $\lim\limits_{n\rightarrow\infty}\zeta(x_n)= \lim\limits_{n\rightarrow\infty}\zeta(y_n)> 0$.
    \end{enumerate}Then the fractal operator, $W$ generated by the IFS $\{X; w_1, w_2,\dots, w_N\},$ has a unique attractor, say $A^*\in\mathcal{H}_f(X)$. Moreover, the iterated sequence $\{W^n(A)\}$ converges to the attractor $A^*$ for any $A\in\mathcal{H}_f(X)$. 
\end{theorem}
\begin{proof}
    It follows from Lemma \ref{lem:4.7} that the fractal operator $W$ is a Proinov-type $\mathcal{Z}$-contraction in the complete metric space $\mathcal{H}_f(X)$. Then the result follows immediately from Theorem \ref{thm:4.2}.
\end{proof}
The\textcolor{white}{a}following example\textcolor{white}{a}will illustrate\textcolor{white}{a}Theorem \ref{thm:4.5}.
\begin{eg}
    Let\textcolor{white}{a}$X=[0, 1]$. Define $q:[0, 1]\times[0, 1]\rightarrow \mathbb{R}$ as 
    \[q(x, y)= \begin{cases}
        8x & if x>y\\
        4y & if x<y\\
        0 & if x=y.
    \end{cases}\] 
    It can be easily verified that $q$ is a $2$-symmetric quasi-metric on $[0, 1]$. Now define $\zeta, \eta:[0, \infty)\rightarrow \mathbb{R}$ as $\zeta(t)= t$ and $\eta(t)=t^2$. Both $\zeta \text{ and }\eta$ satisfy the conditions $(ii)- (iv)$ of the hypothesis. Consider two simulation functions $\xi_1$ and $\xi_2$ defined as $\xi_1(s, t)=\frac{t}{t+1}-s$ and $\xi_2(s, t)= \frac{16t}{t+16}-s$. Clearly, both $\xi_1$ and $\xi_2$ satisfies condition $(i)$ in the hypothesis. Define $w_1, w_2:[0, 1]\rightarrow [0, 1]$ by $w_1(x, y)=\frac{x^3}{66x^2+3}$ and $w_2(x)=\frac{4x^2}{4x^2+1}$. We will prove that both $w_1$ and $w_2$ are $f$-Proinov-type $\mathcal{Z}$-contractions with respect to the simulation functions $\xi_1$ and $\xi_2$ respectively.\\
    First, we consider the function $w_1$ and the simulation function $\xi_1$.\\
    \textbf{Case 1:} If $x>y$, then $q(x, y)=8x$ and $q(w_1(x),w_2(x))=\frac{8x^3}{66x^2+3}$. Then the
\begin{equation*}
\xi_1\left(\zeta\left(q\left(w_1(x),w_1(y)\right)\right), \eta\left(q\left(x,y\right)\right)\right)=\frac{64x^2}{64x^2+1}-\frac{8x^3}{66x^2+5}
\geq\frac{64x^2}{64x^2+1}-\frac{8x^3}{64x^2+1}
\geq 0
\end{equation*}
\textbf{Case 2:} If $x<y$, then $q(x,y)= 4y$ and $q(w_1(x),w_1(y))=\frac{4y^3}{66y^2+3}$. Then,
\begin{equation*}
    \xi_1\left(\zeta\left(q\left(w_1(x),w_1(y)\right)\right),\eta\left(q\left(x,y\right)\right)\right)=\frac{16y^2}{16y^2+1}-\frac{4y^3}{66y^2+3}\geq\frac{16y^2}{16y^2+1}-\frac{4y^3}{16y^2+1}\geq 0.
\end{equation*}
From both cases, it can be observed that the self-mapping $w_1$ is an $f$-Proinov-type $\mathcal{Z}$-contraction on $[0,1]$ with respect to the simulation function $\xi_1$.\\
 Next, we consider the self-mapping $w_2$ and the simulation function $\xi_2$.\\
 \textbf{Case 1:} If $x>y$, then $q(x,y)= 8x$ and $q\left(w_2(x), w_2(y)\right)=\frac{32x^2}{4x^2+1}$. Thus,
 \begin{equation*}
     \xi_2\left(\zeta\left(q\left(w_2(x),w_2(y)\right)\right),\eta\left(q(x,y)\right)\right)=\frac{64x^2}{4x^2+1}-\frac{32x^2}{4x^2+1}\geq 0.
     \end{equation*}
     \textbf{Case 2:} If $x<y$, then $q(x,y)=4y$ and $q\left(w_2(x),w_2(y)\right)=\frac{16y^2}{4y^2+1}$. Then,
     \begin{equation*}
         \xi_2\left(\zeta\left(q\left(w_2(x),w_2(y)\right)\right),\eta\left(q(x,y\right)\right)=\frac{16y^2}{y^2+1}-\frac{16y^2}{4y^2+1}\geq\frac{16y^2}{y^2+1}-\frac{16y^2}{y^2+1}=0.
     \end{equation*}
     Hence, $w_2$ is an $f$-Proinov-type $\mathcal{Z}$-contraction on $[0,1]$ with respect to the simulation function $\xi_2$.\\ Then by Theorem \ref{thm:4.5}, the collection $\left\{[0,1], w_1,w_2\right\}$ forms an IFS. Furthermore, the fractal operator $W$ generated by this IFS is a Proinov-type $\mathcal{Z}$-contraction on the complete metric space $\mathcal{H}_f\left([0,1]\right)$ with\textcolor{white}{a}respect to\textcolor{white}{a}the simulation\textcolor{white}{a}function $\xi(s,t)=\max\left\{\xi_1(s,t),\xi_2(s,t)\right\}= \max\left\{\frac{t}{t+1}-s, \frac{16t}{t+16}-s\right\}$. The Theorem \ref{thm:4.5} also guarantees the existence of a unique attractor of this IFS. Here, we can observe that $w_1\left([0,\frac{1}{2}]\right)=[0,\frac{1}{172}]$ and $w_2\left([0,\frac{1}{2}]\right)=[0,\frac{1}{2}]$. Thus, $W\left([0,\frac{1}{2}]\right)=w_1\left([0,\frac{1}{2}]\right)\bigcup w_2\left([0,\frac{1}{2}]\right)=[0,\frac{1}{2}]$. In addition, we can observe that $[0,\frac{1}{2}]$ is the unique attractor of this IFS.
\end{eg}

\end{document}